\documentclass[11pt, reqno]{amsart}

\author[P.~Leonetti]{Paolo Leonetti}
\address{Universit\`a ``Luigi Bocconi''\\Department of Statistics\\Milan, Italy}
\email{leonetti.paolo@gmail.com}
\urladdr{\url{http://orcid.org/0000-0001-7819-5301}}

\author[C.~Sanna]{Carlo Sanna}
\address{Universit\`a degli Studi di Torino\\Department of Mathematics\\Turin, Italy}
\email{carlo.sanna.dev@gmail.com}
\urladdr{\url{http://orcid.org/0000-0002-2111-7596}}

\keywords{Fibonacci numbers, rank of appearance, greatest common divisor, natural density.}
\subjclass[2010]{Primary: 11B39. Secondary: 11A05, 11N25.}

\title{On the greatest common divisor of $n$ \\ and the $n$th Fibonacci number}

\usepackage{amsmath}
\usepackage{amssymb}
\usepackage{amsthm}
\usepackage[left=3.5cm, right=3.5cm, paperheight=11.8in]{geometry}
\usepackage{hyperref}
\usepackage{fancyhdr}
\usepackage{enumitem}
\usepackage{graphicx}
\usepackage[utf8]{inputenc}

\newtheorem{thm}{Theorem}[section]
\newtheorem{lem}[thm]{Lemma}
\newtheorem{cor}[thm]{Corollary}

\theoremstyle{definition}

\hypersetup{
    pdftitle={On the greatest common divisors of n and the n-th Fibonacci numbers},
    pdfauthor={Paolo Leonetti and Carlo Sanna},
    pdfmenubar=false,
    pdffitwindow=true,
    pdfstartview=FitH,
    colorlinks=true,
    linkcolor=blue,
    citecolor=green,
    urlcolor=cyan
}

\def\lcm{\operatorname{lcm}}
                                 
\begin{document}

\maketitle
\thispagestyle{empty}

\begin{abstract}
Let $\mathcal{A}$ be the set of all integers of the form $\gcd(n, F_n)$, where $n$ is a positive integer and $F_n$ denotes the $n$th Fibonacci number.
We prove that $\#\left(\mathcal{A} \cap [1, x]\right) \gg x / \log x$ for all $x \geq 2$, and that $\mathcal{A}$ has zero asymptotic density. 
Our proofs rely on a recent result of Cubre and Rouse [Proc. Amer. Math. Soc. \textbf{142} (2014), 3771--3785] which gives, for each positive integer $n$, an explicit formula for the density of primes $p$ such that $n$ divides the rank of appearance of $p$, that is, the smallest positive integer $k$ such that $p$ divides $F_k$.
\end{abstract}

\section{Introduction}

Let $(F_n)_{n \geq 1}$ be the sequence of Fibonacci numbers, defined as usual by $F_1 = F_2 = 1$ and $F_{n+2} = F_{n+1} + F_n$, for all positive integers $n$.
Moreover, let $g$ be the arithmetic function defined by $g(n) := \gcd(n, F_n)$, for each positive integer $n$.
The first values of $g$ are listed in OEIS A104714~\cite{A104714}. 

The set $\mathcal{B}$ of fixed points of $g$, i.e., the set of positive integers $n$ such that $n$ divides $F_n$, has been studied by several authors.
For instance, Andr\'{e}-Jeannin~\cite{MR1131414} and Somer~\cite{MR1271392} investigated the arithmetic properties of the elements of $\mathcal{B}$. 
Furthermore, Luca and Tron~\cite{MR3409327} proved that
\begin{equation}\label{equ:Bbound}
\#\mathcal{B}(x) \leq x^{1 - \left(\tfrac1{2} + o(1)\right)\log \log \log x / \log \log x},
\end{equation}
when $x \to +\infty$, and Sanna~\cite{MR3606950} generalized their result to Lucas sequences. 
More generally, the study of the distribution of positive integers $n$ dividing the $n$th term of a linear recurrence has been studied by Alba Gonz\'alez, Luca, Pomerance, and Shparlinski~\cite{MR2928495},
while, Corvaja and Zannier~\cite{MR1918678}, and Sanna~\cite{San_preprint} considered the distribution of positive integers $n$ such that the $n$th term of a linear recurrence divides the $n$th term of another linear recurrence.
Also, it follows from a result of Sanna~\cite{San2_preprint} that the set $g^{-1}(1)$, i.e., the set of positive integers $n$ such that $n$ and $F_n$ are relatively prime, has a positive asymptotic density.

Define $\mathcal{A}:= \{g(n) : n \geq 1\}$.
Note that, in particular, $\mathcal{B} \subseteq \mathcal{A}$. 
The aim of this article is to study the structural properties and the distribution of the elements of $\mathcal{A}$. 
Note that it is not immediately clear whether or not a given positive integer belongs to $\mathcal{A}$. 
To this aim, we provide in \S\ref{sec:prelim} an effective criterion which allows us to enumerate the elements of $\mathcal{A}$, in increasing order, as:
\begin{equation*}
1,\;\; 2,\;\; 5,\;\; 7,\;\; 10,\;\; 12,\;\; 13,\;\; 17,\;\; 24,\;\; 25,\;\; 26,\;\; 29,\;\; 34,\;\; 35,\;\; 36,\;\; 37,\;\; \ldots
\end{equation*}

Our first result is a lower bound for the counting function of $\mathcal{A}$.

\begin{thm}\label{thm:lbound}
$\#\mathcal{A}(x) \gg x / \log x$, for all $x \geq 2$.
\end{thm}

It is worth noting that it follows at once from Theorem~\ref{thm:lbound} and \eqref{equ:Bbound} that $\mathcal{B}$ has zero asymptotic density relative to $\mathcal{A}$ (we omit the details):

\begin{cor}
$\#\mathcal{B}(x) = o(\#\mathcal{A}(x))$, as $x \to +\infty$.
\end{cor}

Our second result is that $\mathcal{A}$ has zero asymptotic density:

\begin{thm}\label{thm:ubound}
$\#\mathcal{A}(x) = o(x)$, as $x \to +\infty$.
\end{thm}

It would be nice to have an effective upper bound for $\#\mathcal{A}(x)$ or, even better, to obtain its asymptotic order of growth.
We leave these as open questions for the interested readers.

\subsection*{Notation}

Throughout, we reserve the letters $p$ and $q$ for prime numbers. 
Moreover, given a set $\mathcal{S}$ of positive integers, we define $\mathcal{S}(x):=\mathcal{S}\cap [1,x]$ for all $x\ge 1$. 
We employ the Landau--Bachmann ``Big Oh'' and ``little oh'' notations $O$ and $o$, as well as the associated Vinogradov symbols $\ll$ and $\gg$.
In particular, all the implied constants are intended to be absolute, unless it is explicitly stated otherwise.

\section{Preliminaries}\label{sec:prelim}

This section is devoted to some preliminary results needed in the later proofs.
For each positive integer $n$, let $z(n)$ be \emph{rank of appearance of $n$} in the sequence of Fibonacci numbers, that is, $z(n)$ is the smallest positive integer $k$ such that $n$ divides $F_k$.
It is well known that $z(n)$ exists.
All the statements in the next lemma are well known, and we will use them implicitly without further mention.

\begin{lem}\label{lem:basic}
For all positive integer $m,n$ and all prime numbers $p$, we have:
\begin{enumerate}[label={\rm (\roman{*})}]
\item $F_m \mid F_n$ whenever $m \mid n$.
\item $m \mid F_n$ if and only if $z(m) \mid n$.
\item $z(m) \mid z(n)$ whenever $m \mid n$.
\item $z(p) \mid p-\left(\frac{p}{5}\right)$, where $\left(\frac{p}{5}\right)$ is a Legendre symbol.
\end{enumerate}
\end{lem}

For each positive integer $n$, define $\ell(n) := \lcm(n, z(n))$. 
The next lemma shows some elementary properties of the functions $g$, $\ell$, $z$, and their relationship with $\mathcal{A}$.

\begin{lem}\label{lem:characterizationA}
For all positive integer $m,n$ and all prime numbers $p$, we have:
\begin{enumerate}[label={\rm (\roman{*})}]
\item \label{item1b} $g(m) \mid g(n)$ whenever $m \mid n$.
\item \label{item2b} $n \mid g(m)$ if and only if $\ell(n) \mid m$.
\item \label{item3b} $n \in \mathcal{A}$ if and only if $n = g(\ell(n))$.
\item \label{item4b} $p \mid n$ whenever $\ell(p) \mid \ell(n)$ and $n \in \mathcal{A}$.
\item \label{item5b} $\ell(p) = p z(p)$ whenever $p \neq 5$, and $\ell(5) = 5$.
\item \label{item6b} $p \in \mathcal{A}$ if $p \neq 3$ and $\ell(q) \nmid z(p)$ for all prime numbers $q$.
\end{enumerate}
\end{lem}
\begin{proof}
Facts \ref{item1b} and \ref{item2b} follow easily from the definitions of $g$ and $\ell$ and the properties of $z$. 
To prove \ref{item3b}, note that $n$ divides both $\ell(n)$ and $F_{\ell(n)}$ hence $n \mid g(\ell(n))$ for all positive integers $n$. 
Conversely, if $n \in \mathcal{A}$, then $n=g(m)$ for some positive integer $m$. 
In~particular, $n\mid g(m)$ which is equivalent to $\ell(n)\mid m$ by \ref{item2b}. 
Therefore $g(\ell(n)) \mid g(m) = n$, thanks to \ref{item1b}, and in conclusion $g(\ell(n)) = n$.
Fact \ref{item4b} follows at once from \ref{item2b} and \ref{item3b}. 

A quick computation shows that $\ell(5) = 5$, while for all prime numbers $p\neq 5$ we have $\gcd(p,z(p))=1$, since $z(p)\mid p \pm 1$, so that $\ell(p)=pz(p)$, and this proves \ref{item5b}. 

Lastly, let us suppose that $p\neq 3$ is a prime number such that $\ell(q) \nmid z(p)$ for all prime numbers $q$.
In particular, $p\neq 5$ since $\ell(5)=z(5)=5$, by \ref{item5b}.
Also, the claim \ref{item6b} is easily seen to hold for $p = 2$.
Hence, let us suppose hereafter that $p \geq 7$. 
Since $z(p) \mid p\pm 1$, it easily follows that $p \mid\mid g(\ell(p))$. 
At this point, if $q \mid g(\ell(p))$ for some prime $q \neq p$, then $\ell(q) \mid \ell(p)=pz(p)$ thanks to \ref{item2b}.
But $\ell(q)\nmid z(p)$, hence $p\mid \ell(q)=\lcm(q,z(q))$ so that $p\mid z(q)\le q+1$. 
Similarly, $q\mid g(\ell(p))\mid \ell(p)$ implies $q\mid z(p) \le p+1$. 
Hence $|p - q| \leq 1$, which is impossible since $p \ge 7$.
Therefore $q\nmid g(\ell(p))$, with the consequence that $p=g(\ell(p))$, i.e., $p \in \mathcal{A}$ by \ref{item3b}. 
This concludes the proof of \ref{item6b}.
\end{proof}

It is worth noting that Lemma~\ref{lem:characterizationA}\ref{item3b} provides an effective criterion to establish whether a given positive integer belongs to $\mathcal{A}$ or not. This is how we evaluated the elements of $\mathcal{A}$ listed in the introduction.

It follows from a result of Lagarias~\cite{MR789184,MR1251907}, that the set of prime numbers $p$ such that $z(p)$ is even has a relative density of $2/3$ in the set of all prime numbers.
Bruckman and Anderson~\cite[Conjecture 3.1]{MR1627443} conjectured, for each positive integer $m$, a formula for the limit 
\begin{equation*}
\zeta(m) := \lim_{x \to +\infty} \frac{\#\{p \leq x : m \mid z(p)\}}{x / \log x} .
\end{equation*}
Their conjecture was proved by Cubre and Rouse~\cite[Theorem~2]{MR3251719}, who obtained the following result.

\begin{thm}\label{thm:cubrerouse}
For each prime number $q$ and each positive integer $e$, we have
\begin{equation*}
\zeta(q^e) = \frac{q^{2-e}}{q^2 - 1} ,
\end{equation*}
while for any positive integer $m$, we have
\begin{equation*}
\zeta(m) = \prod_{q^e \mid\mid m} \zeta(q^e) \cdot \begin{cases} 1 & \text{ if } 10 \nmid m, \\ \tfrac{5}{4}  & \text{ if } m \equiv 10 \bmod {20}, \\ \tfrac1{2} & \text{ if } 20 \mid m .\end{cases}
\end{equation*}
\end{thm}

Note that the arithmetic function $\zeta$ is not multiplicative.
However, the restriction of $\zeta$ to the odd positive integers is multiplicative.
This fact will be useful later.

Let $\varphi$ be the Euler's totient function.
We need the following technical lemma.

\begin{lem}\label{lem:phiellq}
We have
\begin{equation*}
\sum_{q > y} \frac1{\varphi(\ell(q))} \ll \frac1{y^{1/4}} , 
\end{equation*}
for all $y > 0$.
\end{lem}
\begin{proof}

For $\gamma > 0$, put $\mathcal{Q}_\gamma := \{p : z(p) < p^\gamma\}$. 
Clearly,
\begin{equation*}
2^{\# \mathcal{Q}_\gamma(x)} \leq \prod_{p \in \mathcal{Q}_\gamma(x)} p \mid \prod_{n \leq x^\gamma} F_n \leq 2^{\sum_{n \leq x^\gamma} n} \leq 2^{O(x^{2\gamma})} ,
\end{equation*}
from which it follows that $\mathcal{Q}_\gamma(x) \ll x^{2\gamma}$. 

Fix also $\varepsilon \in {]0,1-2\gamma[}$.
For the rest of this proof, all the implied constants may depend on $\gamma$ and $\varepsilon$. 
Since $\varphi(n) \gg n / \log \log n$ for all positive integers $n$ \cite[Ch.~I.5, Theorem~4]{Ten95}, while, by Lemma~\ref{lem:characterizationA}\ref{item5b}, $\ell(q) \ll q^2$ for all prime numbers $q$, we have
\begin{equation}\label{eq:eq1}
\sum_{q > y} \frac1{\varphi(\ell(q))} \ll \sum_{q > y} \frac{\log \log \ell(q)}{\ell(q)}\ll \sum_{q > y} \frac{\log \log q}{\ell(q)} \ll \sum_{q > y} \frac{q^\varepsilon}{\ell(q)} ,
\end{equation}
for all $y > 0$.

On the one hand, again by Lemma~\ref{lem:characterizationA}\ref{item5b}, 
\begin{equation}\label{eq:eq2}
\sum_{\substack{q > y \\ q \notin \mathcal{Q}_\gamma}} \frac{q^{\varepsilon}}{\ell(q)} \ll \sum_{\substack{q > y \\ q \notin \mathcal{Q}_\gamma}} \frac{1}{q^{1-\varepsilon}z(q)} \le \sum_{q>y} \frac{1}{q^{1+\gamma-\varepsilon}} \ll \int_y^{+\infty} \frac{\mathrm{d}t}{t^{1 + \gamma - \varepsilon}}\ll \frac{1}{y^{\gamma-\varepsilon}} .
\end{equation}
On the other hand, by partial summation,
\begin{align}\label{eq:eq3}
\sum_{\substack{q > y \\ q \in \mathcal{Q}_\gamma}} \frac{q^{\varepsilon}}{\ell(q)} &\le \sum_{\substack{q > y \\ q \in \mathcal{Q}_\gamma}}\frac{1}{q^{1-\varepsilon}} = \left.\frac{\#\mathcal{Q}_\gamma(t)}{t^{1-\varepsilon}} \right|_{t=y}^{+\infty} + (1-\varepsilon)\int_y^{+\infty} \frac{\#\mathcal{Q}_\gamma(t)}{t^{2-\varepsilon}} \,\mathrm{d}t \nonumber \\
&\leq \int_y^{+\infty} \frac{\#\mathcal{Q}_\gamma(t)}{t^{2-\varepsilon}} \,\mathrm{d} t \ll \int_y^{+\infty} \frac{\mathrm{d} t}{t^{2-2\gamma-\varepsilon}} \ll \frac{1}{y^{1-2\gamma-\varepsilon}}.
\end{align}
The claim follows by putting together \eqref{eq:eq1}, \eqref{eq:eq2}, and \eqref{eq:eq3}, and by choosing $\gamma=1/3$ and $\varepsilon=1/12$.
\end{proof}

Lastly, for all relatively prime integers $a$ and $m$, define
\begin{equation*}
\pi(x, m, a) := \#\{p \leq x : p \equiv a \bmod m\} .
\end{equation*}
We need the following version of the Brun--Titchmarsh theorem~\cite[Theorem~2]{MR0374060}.

\begin{thm}\label{thm:bruntitchmarsh}
If $a$ and $m$ are relatively prime integers and $m > 0$, then
\begin{equation*}
\pi(x, m, a) < \frac{2 x}{\varphi(m) \log (x / m)} ,
\end{equation*}
for all $x > m$.
\end{thm}

\section{Proof of Theorem~\ref{thm:lbound}}

First, since $1 \in \mathcal{A}$, it is enough to prove the claim only for all sufficiently large $x$.
Let $y > 5$ be a real number to be chosen later.
Define the following sets of primes:
\begin{align*}
\mathcal{P}_1 &:= \{p : q \nmid z(p), \;\forall q \in [3, y]\} , \\
\mathcal{P}_2 &:= \{p : \exists q > y, \; \ell(q) \mid z(p)\} , \\
\mathcal{P} &:= \mathcal{P}_1 \setminus \mathcal{P}_2 .
\end{align*}
We have $\mathcal{P} \subseteq \mathcal{A} \cup \{3\}$.
Indeed, since $3 \mid \ell(2)$ and $q \mid \ell(q)$ for each prime number $q$, it follows easily that if $p \in \mathcal{P}$ then $\ell(q) \nmid z(p)$ for all prime numbers $q$, which, by Lemma~\ref{lem:characterizationA}\ref{item6b}, implies that $p \in \mathcal{A}$ or $p=3$.

Now we give a lower bound for $\#\mathcal{P}_1(x)$.
Let $P_y$ be the product of all prime numbers in~$[3, y]$, and let $\mu$ be the M\"obius function.
By using the inclusion-exclusion principle and Theorem~\ref{thm:cubrerouse}, we get that
\begin{align*}
\lim_{x \to +\infty} \frac{\#\mathcal{P}_1(x)}{x / \log x} &= \lim_{x \to +\infty} \sum_{m \mid P_y} \mu(m) \cdot \frac{\#\{p \leq x : m \mid z(p)\}}{x / \log x} = \sum_{m \mid P_y} \mu(m) \zeta(m) \\ 
 &= \prod_{3 \leq q \leq y} \left(1 - \zeta(q)\right) = \prod_{3 \leq q \leq y} \left(1 - \frac{q}{q^2 - 1}\right) ,
\end{align*}
where we also made use of the fact that the restriction of $\zeta$ to the odd positive integers is multiplicative.
 
As a consequence, for all sufficiently large $x$ depending only on $y$, let say $x \geq x_0(y)$, we have
\begin{equation*}
\#\mathcal{P}_1(x) \geq \frac1{2} \prod_{3 \leq q \leq y} \left(1 - \frac{q}{q^2 - 1}\right) \cdot \frac{x}{\log x} \gg \frac1{\log y} \cdot \frac{x}{\log x} ,
\end{equation*}
where the last inequality follows from Mertens' third theorem~\cite[Ch. I.1, Theorem~11]{Ten95}.

We also need an upper bound for $\#\mathcal{P}_2(x)$.
Since $z(p) \mid p \pm 1$ for all primes $p > 5$, we have
\begin{equation}\label{equ:P2bound1}
\#\mathcal{P}_2(x) \leq \sum_{q > y} \#\{p \leq x : \ell(q) \mid z(p)\} \leq \sum_{q > y} \pi(x, \ell(q), \pm 1) ,
\end{equation}
for all $x > 0$, where, for the sake of brevity, we put
\begin{equation*}
\pi(x, \ell(q), \pm 1) := \pi(x, \ell(q), -1) + \pi(x, \ell(q), 1) .
\end{equation*}
On the one hand, by Theorem~\ref{thm:bruntitchmarsh} and Lemma~\ref{lem:phiellq}, we have
\begin{equation}\label{equ:P2bound2}
\sum_{y < q < x^{1/2}} \pi(x, \ell(q), \pm 1) \ll \sum_{q > y} \frac1{\varphi(\ell(q))} \cdot \frac{x}{\log x} \ll \frac1{y^{1/4}} \cdot \frac{x}{\log x} .
\end{equation}
On the other hand, by the trivial estimate for $\pi(x, \ell(q), \pm 1)$ and Lemma~\ref{lem:phiellq}, we get
\begin{equation}\label{equ:P2bound3}
\sum_{q > x^{1/2}} \pi(x, \ell(q), \pm 1) \ll \sum_{q > x^{1/2}} \frac{x}{\ell(q)} \leq \sum_{q > x^{1/2}} \frac{x}{\varphi(\ell(q))} \ll x^{7/8} .
\end{equation}
Therefore, putting together (\ref{equ:P2bound1}), (\ref{equ:P2bound2}), and (\ref{equ:P2bound3}), we find that
\begin{equation*}
\#\mathcal{P}_2(x) \ll \frac1{y^{1/4}} \cdot \frac{x}{\log x} + x^{7/8} .
\end{equation*}
In conclusion, there exist two absolute constants $c_1, c_2 > 0$ such that
\begin{equation}\label{equ:Albound}
\#\mathcal{A}(x) \gg \#\mathcal{P}(x) \geq \#\mathcal{P}_1(x) - \#\mathcal{P}_2(x) \geq \left(\frac{c_1}{\log y} - \frac{c_2}{y^{1/4}} - \frac{c_2 \log x}{x^{1/8}} \right) \cdot \frac{x}{\log x} ,
\end{equation}
for all $x \geq x_0(y)$.

Finally, we can choose $y$ to be sufficiently large so that
\begin{equation*}
\frac{c_1}{\log y} - \frac{c_2}{y^{1/4}} > 0 .
\end{equation*}
Hence, from (\ref{equ:Albound}) it follows that $\#\mathcal{A}(x) \gg x / \log x$, for all sufficiently large $x$.

\section{Proof of Theorem~\ref{thm:ubound}}

Fix $\varepsilon > 0$ and pick a prime number $q$ such that $1/q < \varepsilon$.
Let $\mathcal{P}$ be the set of prime numbers $p$ such that $\ell(q) \mid z(p)$.
By Theorem~\ref{thm:cubrerouse}, we know that $\mathcal{P}$ has a positive relative density in the set of primes.
As a consequence, we can pick a sufficiently large $y > 0$ so that
\begin{equation*}
\prod_{p \in \mathcal{P}(y)} \left(1 - \frac1{p}\right) < \varepsilon .
\end{equation*}
Let $\mathcal{B}$ be the set of positive integers without prime factors in $\mathcal{P}(y)$.
We split $\mathcal{A}$ into two subsets: $\mathcal{A}_1 := \mathcal{A} \cap \mathcal{B}$ and $\mathcal{A}_2 := \mathcal{A} \setminus \mathcal{A}_1$.
If $n \in \mathcal{A}_2$ then $n$ has a prime factor $p$ such that $\ell(q) \mid z(p)$.
Hence, $\ell(q) \mid \ell(n)$ and, by Lemma~\ref{lem:characterizationA}\ref{item4b}, we get that $q \mid n$, so all the elements of $\mathcal{A}_2$ are multiples of $q$.
In conclusion,
\begin{equation*}
\limsup_{x \to +\infty} \frac{\#\mathcal{A}(x)}{x} \leq \limsup_{x \to +\infty} \frac{\#\mathcal{A}_1(x)}{x} + \limsup_{x \to +\infty} \frac{\#\mathcal{A}_2(x)}{x} \leq \prod_{p \in \mathcal{P}(y)} \left(1 - \frac1{p}\right) + \frac1{q} < 2\varepsilon ,
\end{equation*}
and, by the arbitrariness of $\varepsilon$, it follows that $\mathcal{A}$ has zero asymptotic density.

\providecommand{\bysame}{\leavevmode\hbox to3em{\hrulefill}\thinspace}
\providecommand{\MR}{\relax\ifhmode\unskip\space\fi MR }
\providecommand{\MRhref}[2]{%
  \href{http://www.ams.org/mathscinet-getitem?mr=#1}{#2}
}
\providecommand{\href}[2]{#2}

\end{document}